\documentclass[12pt,twoside,reqno]{article}
 \usepackage[T1]{fontenc}

\usepackage{latexsym}
\usepackage{amssymb,amsmath,amsthm,bm,setspace}
\usepackage{pdfsync}	
\usepackage{color}
\usepackage{graphicx}

\newcommand {\E}{\mathrm{E}}

\renewcommand{\d}{\text{\rm d}}

\newcommand{\sA}{\mathcal{A}}

\newcommand{\lip}{\mathrm{Lip}}

\emergencystretch=\maxdimen
\newtheorem{stat}{Statement}[section]

\newtheorem{corollary}[stat]{Corollary}

\newtheorem{theorem}[stat]{Theorem}

\newtheorem{lemma}[stat]{Lemma}
\theoremstyle{definition}
\newtheorem{dfn}[stat]{Definition}\newtheorem{remark}[stat]{Remark}

\newtheorem{condition}[stat]{Condition}

\numberwithin{equation}{section}
\title{Asymptotic behavior of solution  and non-existence of global solution to a class of conformable time-fractional stochastic equation}
\begin{document}
\author{Erkan Nane$^{1,*}$\\
{\tt ezn0001@auburn.edu }
\and
Eze R. Nwaeze$^{2,+}$\\
{\tt  enwaeze@alasu.edu}
\and
McSylvester Ejighikeme Omaba$^{3,**}$\\
{\tt mcomaba@uhb.edu.sa}}
\date{$^1$ Department of Mathematics and Statistics, Auburn University, Auburn, AL 36849, USA\\[0.3cm]
$^2$Department of Mathematics and Computer Science, Alabama State University, Montgomery, AL 36101, USA,\\
$^3$ Department of Mathematics, College of Science, University of Hafr Al Batin P. O Box 1803 Hafr Al Batin 31991, KSA\\[0.3cm]
}

\maketitle

\begin{abstract}Consider the following class of conformable time-fractional  stochastic equation
$$T_{\alpha,t}^a u(x,t)=\lambda\sigma(u(x,t))\dot{W}_t,\,\,\,\,x\in\mathbb{R},\,t\in[a,\infty), \,\,0<\alpha<1,$$ with a non-random initial condition $u(x,0)=u_0(x),\,x\in\mathbb{R}$  assumed to be non-negative and bounded, $T_{\alpha,t}^a$ is a conformable time - fractional derivative, $\sigma:\mathbb{R}\rightarrow\mathbb{R}$ is globally Lipschitz continuous, $\dot{W}_t$  a generalized derivative of Wiener process and $\lambda>0$ is the noise level.  Given some precise and suitable conditions on the non-random initial function, we study the asymptotic behaviour of the solution with respect to the time parameter $t$ and the noise level parameter $\lambda$. We also show that when the non-linear term $\sigma$ grows faster than linear, the energy of the solution blows-up at finite time for all $\alpha\in (0,1)$.
\end{abstract}
{\bf 2010 MSC No:} 26A33,\, 34A08,\, 60H15,\,82B44.

\maketitle
\section{ Introduction}
\indent  It has been shown that most results on stochastic partial differential equations assume a global Lipschitz continuity condition on the mutliplicative non-linearity term $\sigma$, see \cite{Foondun,Foondun1,Foondun2, Omaba, Omaba1, Omaba2}. Many results  exist on asymptotic behaviour of solutions and non-existence of global solutions for different classes of fractional stochastic heat equations, see \cite{Foondun2,Foondun3,Omaba4} and their references.
In this paper, we study the long-term behaviour and global non-existence of solution to a class of conformable time- fractional stochastic differential equation. 

\indent The use of fractional derivative  has now received an increasingly high interest and  attention due to its physical and modelling applications in Science, Engineering and Mathematics, see \cite{Zhao} and the references threin. There are various definitions and generalizations of the fractional derivative, not limited to the Riemann--Liouville fractional derivative and Caputo--Dzhrbashyan fractional derivative, with their respective limitations, see \cite{Cenesiz}. One of the limitations of the above two fractional derivatives, is that they do not satisfy the classical chain rule; hence, the need for a better definition of fractional derivative. In \cite{Khalil}, Khalil and his co-authors, introduced a  new well-behaved definition of a fractional derivative known as the conformable fractional derivative, satisfying the usual chain rule, the Rolle's and the mean value theorem, conformable integration by parts, fractional power series expansion and the conformal fractional derivative of a constant real function is zero, etc. These properties  have  given rise to a new research direction. Thus conformable fractional derivative is a natural extension of the classical derivative (since it can be expressed as a first derivative multiplied by a fractional factor or power in some cases) and has many advantages over other fractional derivatives as mentioned above.

\indent The conformable fractional derivative combines the best characteristics of well-known fractional derivatives,  seems more appropriate to describe the behaviour of classical viscoelastic models under interval uncertainty, see \cite{Salahshour} and gives models that agree and are consistent with experimental data, see \cite{Al-Refai}. Conformable fractional derivatives are applied in certain classes of conformable differentiable linear systems subject to impulsive effects and establish quantitative behaviour of the nontrivial solutions (stability, disconjugacy, etc), see \cite{Birgani,Gholami}; and used to develop the Swartzendruber model for description of non-Darcian flow in porous media \cite{Birgani,Yang}. Conformable derivatives are also used to solve approximate long water wave equation with conformable fractional derivative and conformable fractional differential equations via radial basis functions collocation method \cite{Kaplan,Usta}. \\

Though there have been  significant contributions in the study of class of stochastic heat equations with Riemann--Liouville (R--L) and  Caputo--Dzhrbashyan (C--D) fractional derivatives \cite{Foondun,Foondun1, Foondun2,Mijena,Omaba,Omaba1,Omaba2,Walsh},  but not much has been done in the study of stochastic Cauchy equation with conformable fractional derivative, see \cite{Cenesiz}, see also \cite{Meng, Zhong} for  recent studies on  conformable fractional differential equations.

\indent The difference in studying stochastic Cauchy equation with conformable fractional derivative is that there is no singular kernel of the form $(t-s)^{-\alpha}$ generated for $0<\alpha<1$ which reflect the nonlocality and the memory in the fractional operator as in the case of R--L and C--D fractional derivatives. Thus, despite the fact that  it does not possess or satisfy a semigroup property unlike the  R--L and C--D fractional operators that have well-behaved semigroup properties, the use of conformable fractional derivative makes the study of the theory of fractional differential equation easier, see \cite{Abdeljawad,Zhong}.

Due to the above and many other  advantages/importance of conformable derivative, it has become increasingly applicable to the study of both initial and boundary value problems.
 Souahi et al. in \cite{Souahi} studied the following equation as a possible  model for some nonlinear control system
  \begin{equation}\label{eqn01}T_{\alpha,t}^a u(x,t)=f(u(x,t)),\,\,\,\,x\in\mathbb{R},\,0<a<t<\infty,\,\,0<\alpha<1,
\end{equation}
with an initial condition $u(x,0)=u_0(x),\,\,x\in\mathbb{R}$; where $T_{\alpha,t}^a$ is a conformable fractional derivative, $f:\mathbb{R}\rightarrow\mathbb{R}$ is a Lipschitz continuous function.
  They studied the stability and asymptotic stability of equation \eqref{eqn01}, and presented some illustrative
examples to analyse some nonlinear control system in the conformable sense. Asawasamrit et al. in \cite{Asawasamrit} studied the existence of solutions to the boundary value problems for some specific conformable fractional differential equations. In this work  we wish to extend the study of conformable derivative to stochastic equations  and  therefore consider the following class of conformable fractional stochastic equation:
\begin{equation}\label{eqn1}T_{\alpha,t}^a u(x,t)=\lambda\sigma(u(x,t))\dot{W}_t,\,\,\,\,x\in\mathbb{R},\,0<a<t<\infty,\,\,0<\alpha<1,
\end{equation}
with an initial condition $u(x,0)=u_0(x),\,\,x\in\mathbb{R}$; where $T_{\alpha,t}^a$ is a conformable fractional derivative, $\sigma:\mathbb{R}\rightarrow\mathbb{R}$ is a Lipschitz continuous function, $\dot{W}_t$ is a generalized derivative of Wiener process (Gaussian white noise) and $\lambda>0$ is the noise level.
We consequently give the lower moment growth bound and explore the asymptotic behaviour and long-term properties of the solution. We also show that the solution fails to exist globally when the non-linearity term grows faster than linear; thus  our solutions blow up in some sense.

The paper is outlined as follows. Section \ref{200} contains the definitions of basic concepts used; the problem formulation, the main results and their proofs are given in Section \ref{300}, and Section \ref{400} houses a brief conclusion of the paper.

\section{Preliminaries}\label{200}
 For the definition and more  details on conformable fractional derivatives, see \cite{Abdeljawad,Cenesiz, Iyiola, Khalil,Zhao, Zhong} and their references. Let $\mathcal{D}:=[a,\infty)\times\mathbb{R}$.
\begin{dfn} The conformable time-fractional derivative starting from $a$ of a function $u:\mathcal{D}\rightarrow\mathbb{R}$ of order $\alpha$ is
$$T_{\alpha,t}^a u(x,t)=\lim_{\epsilon\rightarrow 0}\frac{u(x,t+\epsilon(t-a)^{1-\alpha})-u(x,t)}{\epsilon}.$$
If $T_{\alpha,t}^a u(x,t)$ exists on $(a,b)$, then $$T_{\alpha,t}^a u(x,a)=\lim_{t\rightarrow a}T_{\alpha,t}^a u(x,t)$$ and if $u_t(x,t)$ exists then $$T_{\alpha,t}^a u(x,t)=(t-a)^{1-\alpha}u_t(x,t).$$
\end{dfn}
\begin{dfn}The fractional integral starting from $a$ of a function $u:\mathcal{D}\rightarrow\mathbb{R}$ of order $\alpha$ is
$$I_{\alpha,t}^a u(x,t)=\int_a^t (s-a)^{\alpha-1}u(x,s)d s.$$
\end{dfn}
In \cite{Zhong}, the authors considered the following problem
\begin{equation}\label{eqn:Com}
T_{\alpha}^a y(t)=f(t,y(t)),\,t\in [a,\infty),\,0<\alpha<1
\end{equation}
 subject to the initial conditions $y(a)=y_a$ or $y(a)+g(a)=y_a$
where $T_{\alpha}^a y(t)$ denotes the conformable fractional derivative starting from $a$ of a function $y$ of order $\alpha$,
$f :[a,\infty) \times \mathbb{R}\rightarrow \mathbb{R}$ is continuous and $g$ is a given functional defined on an appropriate space of functions.
\begin{lemma}[\cite{Zhong}] If  $f$ is Lipschitz continuous, then a function $y$ in $C([a,b])$ is a solution of the initial value problem \eqref{eqn:Com} if and only if it is a continuous solution of the following integral equation
$$y(t)=y_a+I^a_{\alpha}f(t,y(t)),\,t\in [a,b].$$
\end{lemma}
\begin{remark}
The existence and uniqueness result was proved by Banach's contraction principle on $C(I),\,\,I=[a,b],$ equipped with an appropriate weighted maximum norm.  Zhong  and Wang thus defined a function $e(t)$ by
$$e(t)=e^{-\beta\frac{(t-a)^\alpha}{\alpha}},\,\,\beta\in(L,\infty),\,L>0,$$
and for $y\in C(I),$ it was given that
$\|y\|_\beta=\|e(.)y(.)\|, $ where $\|.\|$ denotes a maximum norm on $C(I)$ with $(C(I),\|.\|_\beta)$ a Banach space, (see  \cite{Zhong} Theorem 3.1).
\end{remark}
\begin{lemma}[\cite{Abdeljawad,Khalil}]Assume that $f:(a,\infty)\rightarrow\mathbb{R}$ is continuous and $0<\alpha\leq 1$. Then, for all $t>a$ we have $$ T^a_\alpha I^a_{\alpha,t}f(t)=f(t).$$
\end{lemma}
\begin{lemma}[\cite{Abdeljawad}]Let $f:(a,b)\rightarrow\mathbb{R}$ be differentiable and $0<\alpha\leq 1$. Then, for all $t>a$ we have $$I^a_{\alpha,t} T^a_\alpha(f)(t)=f(t)-f(a).$$\end{lemma}
Next theorem gives a fractional version of Gronwall inequality
\begin{theorem}[\cite{Abdeljawad}] Let $r$ be a continuous, nonnegative function on an interval $J=[a,b]$ and $\delta$ and $k$ be nonnegative constants such that $$r(t)\leq \delta +k\int_a^t  r(s)(s-a)^{\alpha-1}ds\,\,\,(t\in J).$$
Then for all $t\in J$, $r(t)\leq \delta e^{k\frac{(t-a)^\alpha}{\alpha}}.$
\end{theorem}
\begin{remark}The proof uses the fact that  the function $K(t)=e^{-k\frac{(t-a)^\alpha}{\alpha}}$  solves the following conformable differential equation $T^a_\alpha K(t)=-kK(t).$
\end{remark}
Now, for a definition of  generalized derivative for a deterministic function $w(t)$:
\begin{dfn}Given that $g(t)$ is any smooth and compactly supported function, we define the generalized derivative $\dot{w}(t)$ of $w(t)$ (not necessarily differentiable) as
$$\int_0^\infty g(t)\dot{w}(t)d t=-\int_0^\infty \dot{g}(t)w(t)d t.$$
Similarly, the generalized derivative $\dot{W}_t$ of Wiener process with a smooth function $g(t)$ as follows:
$$\int_0^t g(s)\dot{W}_sd s=g(t)W_t-\int_0^t \dot{g}(s)W_sd s.$$
\end{dfn}

\section{Main Results}\label{300}
For the existence and uniqueness result, we assume that $\sigma$ is globally Lipschitz:
\begin{condition}\label{cond:E-U}
There exists a finite positive constant, $\lip_\sigma$ such that for all $x,\,y\in\mathbb{R}$, we have
\begin{equation*}\label{cond:sigma}
 |\sigma(x)-\sigma(y)|\leq \lip_\sigma|x-y|,
\end{equation*}
with $\sigma(0)=0$ for convenience.
\end{condition}

Define the following norm based on $L^2(\mathbb{P})$ norm $$\|u\|_{2,\alpha,\beta}:=\bigg\{\sup_{t\geq a}\sup_{x\in\mathbb{R}}e^{-\frac{\beta}{\alpha}(t-a)^\alpha}\mathbb{E}|u(x,t)|^2\bigg\}^{1/2}=\bigg\{\sup_{t\geq a}\sup_{x\in\mathbb{R}}e(t)\mathbb{E}|u(x,t)|^2\bigg\}^{1/2},$$ with $e(t):=e^{-\frac{\beta}{\alpha}(t-a)^\alpha}$.  Following  similar ideas in \cite{Abdeljawad, Omaba3,Zhong}, we give the following results:
\begin{lemma}Given that condition \ref{cond:E-U1} below holds, then a function $u$ in $L^2(\mathbb{P})$ is a solution of Equation \eqref{eqn1} if and only if it is a solution of the integral equation $$u(x,t)=u_0(x)+I_{\alpha,t}^a[\lambda\sigma(u(x,t))\dot{W}_t],$$ with $T_{\alpha,t}^a I_{\alpha,t}^a[u(x,t)]=u(x,t)$ for a continuous function $u(x,t)$.
\end{lemma}

Thus, the  mild solution to Equation \eqref{eqn1} is given as follows
\begin{eqnarray}
u(x,t)&=&u_0(x)+\lambda\int_a^t (s-a)^{\alpha-1}\sigma(u(x,s))\dot{W}_s d s\nonumber\\
&=&u_0(x)+\lambda\int_a^t (s-a)^{\alpha-1}\sigma(u(x,s)) d W_s.\label{equation-mild}
\end{eqnarray}

The next condition on $u$  follows from the assumption of Lipschitz continuity of $\sigma$:
\begin{condition}\label{cond:E-U1}The random solution $u:\mathcal{D}\rightarrow\mathbb{R}$ is $L^2$ - continuous ( or continuous in probability).
\end{condition}

\begin{theorem}Suppose that $\alpha>\frac{1}{2}$,  and   $\beta>{(\lambda\lip_\sigma)^2}$ for positive constants $\lip_\sigma$ and $\lambda$. Suppose also that  both Conditions \ref{cond:E-U} and \ref{cond:E-U1} hold. Then there exists   a solution $u$ of equation \eqref{eqn1} that is unique up to modification.
\end{theorem}

The proof of the theorem is based on the following lemmas. Now define the operator
$$\sA u(x,t)=u_0(x)+\lambda\int_a^t (s-a)^{\alpha-1}\sigma(u(x,s))d W_s,$$ and the fixed point of the operator $\sA$ gives the solution of Equation \eqref{eqn1}. We now state the following lemmas, see also $\cite{Omaba3}.$

\begin{lemma}\label{lem:exists}Let $\alpha>1/2$. Suppose $u$ is a predictable random solution such that $\|u\|_{2,2\alpha-1,\beta}<\infty$ and  Conditions  \ref{cond:E-U} and \ref{cond:E-U1} hold. Then for $\beta>0$,
$$\|\sA u\|^2_{2,2\alpha-1,\beta}\leq c_1+\frac{\lambda^2\lip_\sigma^2}{\beta}\|u\|^2_{2,2\alpha-1,\beta}.$$
\end{lemma}

\begin{proof}Since $u_0$ is bounded,  by It\^o isometry  we  have
\begin{eqnarray*}
\mathbb{E}|\sA u(x,t)|^2&\leq& c_1+\lambda^2\lip_\sigma^2\int_a^t (s-a)^{2(\alpha-1)}\mathbb{E}|u(x,s)|^2d s.
\end{eqnarray*}
Let $a(t)=\exp(-\frac{\beta}{2\alpha-1}(t-a)^{2\alpha-1})$ and multiply both sides of the equation above  by $a(t)$ to obtain
\begin{eqnarray*}
a(t)\mathbb{E}|\sA u(x,t)|^2&\leq& c_1a(t)\\&+&\lambda^2\lip_\sigma^2 a(t)\int_a^t (s-a)^{2(\alpha-1)}a^{-1}(s)a(s)\mathbb{E}|u(x,s)|^2d s\\
&\leq& c_1a(t)+\lambda^2\lip_\sigma^2a(t)\|u\|^2_{2,2\alpha-1,\beta}\int_a^t (s-a)^{2(\alpha-1)}a^{-1}(s)d s\\
&\leq& c_1a(t)+\lambda^2\lip_\sigma^2a(t)\|u\|^2_{2,2\alpha-1,\beta}I^a_{2\alpha-1,t}a^{-1}(t),\\
&=& c_1a(t)+\frac{\lambda^2\lip_\sigma^2}{\beta}\|u\|^2_{2,2\alpha-1,\beta}a(t)\big(a^{-1}(t)-1\big)\\
&=& c_1+\frac{\lambda^2\lip_\sigma^2}{\beta}\|u\|^2_{2,2\alpha-1,\beta},
\end{eqnarray*}
where we used the fact that $2\alpha-1>0$ (i.e. $\alpha>\frac{1}{2}$) and  $a(t)\leq1,\,\,a\leq t$  (i.e $0\leq t-a\Rightarrow -\frac{\beta}{2\alpha-1}(t-a)^{2\alpha-1}\leq 0$). Thus taking $\sup$ over $t\in[a,\infty)$ and $x\in\mathbb{R}$  we have
\begin{eqnarray*}
\|\sA u\|^2_{2,2\alpha-1,\beta}&\leq& c_1+\frac{\lambda^2\lip_\sigma^2}{\beta}\|u\|^2_{2,2\alpha-1,\beta}.
\end{eqnarray*}

\end{proof}

\begin{remark}The following estimate was used in the proof of the Lemma \ref{lem:exists}
\begin{eqnarray*}
I^a_{2\alpha-1,t}a^{-1}(t)&=&\int_a^t (s-a)^{2(\alpha-1)}e^{\frac{\beta}{2\alpha-1}(s-a)^{2\alpha-1}}d s\\&=&\frac{1}{\beta}\int_0^{\frac{\beta}{2\alpha-1}(t-a)^{2\alpha-1}}e^y dy=\frac{1}{\beta}\big[a^{-1}(t)-1\big].
\end{eqnarray*}
\end{remark}
\begin{lemma}Suppose $u$ and $v$ are predictable random solutions such that $\|u\|_{2,2\alpha-1,\beta}+\|v\|_{2,2\alpha-1,\beta}<\infty$ and  Conditions  \ref{cond:E-U} and \ref{cond:E-U1} hold. Then for $\beta>0$,
$$\|\sA u-\sA v\|^2_{2,2\alpha-1,\beta}\leq \frac{\lambda^2\lip_\sigma^2}{\beta}\|u-v\|^2_{2,2\alpha-1,\beta}.$$
\end{lemma}

\begin{remark}\label{remark-upper-bounds}
Taking $\beta$ large enough, by fixed point theorem, we have $u(x,t)=\sA u(x,t)$ and
\begin{eqnarray*}
\|u\|^2_{2,2\alpha-1,\beta}=\|\sA u\|^2_{2,2\alpha-1,\beta}\leq c_1+\frac{\lambda^2\lip_\sigma^2}{\beta}\|u\|^2_{2,2\alpha-1,\beta}
\end{eqnarray*}
which follows that $$\|u\|^2_{2,2\alpha-1,\beta}\big[1-\frac{\lambda^2\lip_\sigma^2}{\beta}\big]\leq c_1\Rightarrow \|u\|_{2,2\alpha-1,\beta}<\infty \Leftrightarrow \beta>(\lambda\lip_\sigma)^2.$$
Similarly,
$$\|u-v\|^2_{2,2\alpha-1,\beta}=\|\sA u-\sA v\|^2_{2,2\alpha-1,\beta}\leq \frac{\lambda^2\lip_\sigma^2}{\beta}\|u-v\|^2_{2,2\alpha-1,\beta},$$ thus $\|u-v\|^2_{2,2\alpha-1,\beta}\big[1-\frac{\lambda^2\lip_\sigma^2}{\beta}\big]\leq 0$ and therefore $\|u-v\|_{2,2\alpha-1,\beta}<0$ if and only if $\beta>{(\lambda\lip_\sigma)^2}.$ The existence and uniqueness result follows by Banach's contraction principle.
\end{remark}

The upper growth moment bound (upper bound estimate) on the solution was given with the assumption that the initial condition is bounded above as follows:
\begin{theorem}\label{growth-bound1} Given that Conditions \ref{cond:E-U} and \ref{cond:E-U1} hold, then for $t\in[a,\infty)$ we have
$$\sup_{x\in\mathbb{R}}\mathbb{E}|u(x,t)|^2\leq c_1\exp\big(c_2(t-a)^{2\alpha-1}\big),$$ for some positive constants $c_1$ and $c_2=\frac{\lambda^2\lip_\sigma^2}{2\alpha-1},\,\,\alpha>\frac{1}{2}.$
\end{theorem}

\begin{proof} 
Assume that the initial condition $u_0(x)$ is bounded, then by It\^{o} isometry, we have
\begin{eqnarray*}
\mathbb{E}|u(x,t)|^2&\leq& |u_0(x)|^2+\int_a^t (s-a)^{2(\alpha-1)}\mathbb{E}|\sigma(u(x,s))|^2d s\\
&\leq& c_1+\lambda^2\lip_\sigma^2\int_a^t (s-a)^{2(\alpha-1)}\E|u(x,s)|^2\d s.
\end{eqnarray*}
Let $g(t):=\displaystyle\sup_{x\in\mathbb{R}}\mathbb{E}|u(x,t)|^2$, then by Gronwall's inequality,
\begin{eqnarray*}
g(t)&\leq& c_1+\lambda^2\lip_\sigma^2\int_a^t (s-a)^{2(\alpha-1)}g(s)d s\\
&\leq& c_1\exp\bigg[\lambda^2\lip_\sigma^2\int_a^t (s-a)^{2(\alpha-1)}d s\bigg]\\
&=&c_1\exp\bigg[\frac{\lambda^2\lip_\sigma^2}{2\alpha-1}(t-a)^{2\alpha-1}\bigg]
\end{eqnarray*}and the result follows.
\end{proof}
\subsection{Lower Moment Bound}
Similarly, we have the lower bound estimate by assuming that the initial condition $u_0(x)>c_3$ for some positive constant $c_3$ and
\begin{condition}\label{cond:E-U2}
There exist a finite positive constant, $L_\sigma$ such that for all $x\in\mathbb{R}$, we have
\begin{equation*}\label{cond:sigma}
 |\sigma(x)|\geq L_\sigma|x|.
\end{equation*}
\end{condition}Thus
\begin{theorem}\label{growth-bound2} Given that Conditions \ref{cond:E-U1} and \ref{cond:E-U2} hold, then for $t\in[a,\infty)$ we have
\begin{equation}\label{powers1}
\inf_{x\in\mathbb{R}}\mathbb{E}|u(x,t)|^2\geq c_3\exp\big(c_4(t-a)^{2\alpha-1}\big)
\end{equation}
for some positive constants $c_3$ and $c_4=\frac{\lambda^2 L_\sigma^2}{2\alpha-1},\,\,\alpha>\frac{1}{2}.$
\end{theorem}

\begin{proof}[Proof of Theorem \ref{growth-bound2}]Assume that the initial condition $u_0(x)$ is bounded below, then by It\^{o} isometry, we have
\begin{eqnarray*}
\mathbb{E}|u(x,t)|^2&\geq& |u_0(x)|^2+\lambda^2\int_a^t (s-a)^{2(\alpha-1)}\mathbb{E}|\sigma(u(x,s))|^2d s\\
&\geq& c_3+\lambda^2L_\sigma^2\int_a^t (s-a)^{2(\alpha-1)}\mathbb{E}|u(x,s)|^2\,d s.
\end{eqnarray*}
Let $f(t):=\displaystyle\inf_{x\in\mathbb{R}}\mathbb{E}|u(x,t)|^2$. Then we have 
\begin{eqnarray}
f(t)\geq c_3+\lambda^2L_\sigma^2\int_a^t (s-a)^{2(\alpha-1)}f(s)d s.\label{renewal-ineq}
\end{eqnarray}
Thus solving the differential equation $\dot{f}(t)=\lambda^2L_\sigma^2 (t-a)^{2(\alpha-1)}f(t)$ with $f(a)=c_3$, we have  $f(t)= c_3\exp\bigg[\frac{\lambda^2L_\sigma^2}{2\alpha-1}(t-a)^{2\alpha-1}\bigg]$  and therefore $f(t)\geq c_3\exp\bigg[\frac{\lambda^2L_\sigma^2}{2\alpha-1}(t-a)^{2\alpha-1}\bigg]$ and the result follows.
\end{proof}

 We now give some  immediate consequences of the above Theorem \ref{growth-bound1} and Theorem \ref{growth-bound2}.
First is the long time behaviour of the energy solution, which says that the rate of growth depends on the operator and the noise term (the noise level) as time $t$ grows to infinity:
 \begin{corollary}\label{asymptotics-in-t}
  Suppose that $\alpha>\frac{ 1}{2}$ and  that  positive constants $L_\sigma,\,\lip_\sigma$  are as in Theorems \ref{growth-bound1} and \ref{growth-bound2}. Then for all $x\in \mathbb{R}$ and $\lambda>0$
  \begin{equation*}\label{eqn:t}
\frac{(\lambda L_\sigma)^2}{2\alpha-1}\leq \liminf _{t\rightarrow \infty}\frac{\log\mathbb{E}|u(x,t)|^2}{(t-a)^{2\alpha-1}}\leq\limsup_{t\rightarrow\infty}\frac{\log\mathbb{E}|u(x,t)|^2}{(t-a)^{2\alpha-1}}\leq \frac{(\lambda\lip_\sigma)^2}{2\alpha-1}.
 \end{equation*}
 When $\lip_\sigma=L_\sigma=L$ we have
  $$
  \lim_{t\rightarrow\infty}\frac{\log\mathbb{E}|u(x,t)|^2}{(t-a)^{2\alpha-1}}= \frac{(\lambda L)^2}{2\alpha-1}.
  $$
  \end{corollary}

The next result gives the rate of growth of the second moment with respect to the  noise parameter $\lambda$.

 \begin{corollary}
 Suppose that  $\alpha>\frac{ 1}{2}$ and that  positive constants $L_\sigma,\,\lip_\sigma$  are as in Theorems \ref{growth-bound1} and \ref{growth-bound2}. Then for all $x\in \mathbb{R}$ and $t>a$

 \begin{equation*}\label{eqn:lambda}
 \frac{L_\sigma^2 (t-a)^{2\alpha-1}}{2\alpha-1}\leq\liminf_{\lambda\rightarrow\infty}\frac{\log\mathbb{E}|u(x,t)|^2}{\lambda^2}\leq\limsup_{\lambda\rightarrow\infty}\frac{\log\mathbb{E}|u(x,t)|^2}{\lambda^2}\leq \frac{\lip_\sigma^2 (t-a)^{2\alpha-1}}{2\alpha-1}.
 \end{equation*}
 When  $\lip_\sigma=L_\sigma=L$ we get all $x\in \mathbb{R}$ and $t>a$
 \begin{equation*}\label{eqn:lambda}
\limsup_{\lambda\rightarrow\infty}\frac{\log\mathbb{E}|u(x,t)|^2}{\lambda^2}= \frac{L^2 (t-a)^{2\alpha-1}}{2\alpha-1}.
 \end{equation*}
 \end{corollary}
 
\subsection{Global non-existence of solution}
We show that if the function $\sigma$ grows faster than linear, then  the second moment $\mathbb{E}|u(x,t)|^2$  of the solution to $\eqref{eqn1}$ ceases to exist for all time $t$.
Now suppose that instead of Condition \ref{cond:E-U2}, we have the following condition.
\begin{condition}\label{cond:E-U3}
There exist a finite positive constant, $L_\sigma$ such that for all $x\in\mathbb{R}$, we have
\begin{equation*}\label{cond:sigma}
 |\sigma(x)|\geq L_\sigma|x|^{\beta},\,\,\beta>1.
\end{equation*}
\end{condition}

\begin{theorem}\label{thm:blow}
Suppose that Conditions \ref{cond:E-U1} and \ref{cond:E-U3} are in force. Then there does not exist a solution to Equation \eqref{eqn1} for all $\alpha<\frac{1}{2}.$
\end{theorem}
\begin{proof}
Assume that the initial condition $u_0(x)>c_5$, then by Condition \ref{cond:E-U3}, and applying Jensen's inequality we have
\begin{eqnarray*}
\mathbb{E}|u(x,t)|^2&\geq& |u_0(x)|^2+\lambda^2\int_a^t (s-a)^{2(\alpha-1)}\mathbb{E}|\sigma(u(x,s))|^2d s\\
&\geq& c_5^2+\lambda^2L_\sigma^2\int_a^t (s-a)^{2(\alpha-1)}\mathbb{E}|u(x,s)|^{2(\beta)}d s\\
&\geq& c_6+\lambda^2L_\sigma^2\int_a^t (s-a)^{2(\alpha-1)}\big(\inf_{x\in\mathbb{R}}\mathbb{E}|u(x,s)|^2\big)^{\beta}ds.
\end{eqnarray*}
Let $f(t):=\displaystyle\inf_{x\in\mathbb{R}}\mathbb{E}|u(x,t)|^2$. Then it follows that
\begin{eqnarray*}
f(t)&\geq& c_6+\lambda^2L_\sigma^2\int_a^t (s-a)^{2(\alpha-1)}f^{\beta}(s)d s\\
&=& c_6+\lambda^2L_\sigma^2\int_a^t (s-a)^{2(\alpha-1)-\beta(2\alpha-1)}\bigg((s-a)^{(2\alpha-1)}f(s)\bigg)^\beta d s.
\end{eqnarray*}
Assume that $T>a. $ Multiply through by $(t-a)^{(2\alpha-1)}$ and take $\alpha<\frac{1}{2}$ such that $2\alpha-1<0$, then
\begin{eqnarray*}
y(t)&\geq&  c_6(t-a)^{(2\alpha-1)}+\lambda^2L_\sigma^2(t-a)^{(2\alpha-1)}\int_a^t (s-a)^{2(\alpha-1)-\beta(2\alpha-1)}y^\beta(s) d s\\
&\geq& c_6(T-a)^{(2\alpha-1)}+\lambda^2L_\sigma^2(T-a)^{(2\alpha-1)}\int_a^t (s-a)^{2(\alpha-1)-\beta(2\alpha-1)}y^\beta(s) d s
\end{eqnarray*}
since $a\leq t\leq T\implies 0\leq t-a\leq T-a$  from  which it  follows that\\  $$(t-a)^{(2\alpha-1)}\geq (T-a)^{(2\alpha-1)} \quad\mbox{with}\quad y(t)=(t-a)^{(2\alpha-1)}f(t).$$ \\

Solving the differential equation $$\dot{y}(t)=\lambda^2L_\sigma^2(T-a)^{(2\alpha-1)} (t-a)^{2(\alpha-1)-\beta(2\alpha-1)}y^\beta(t)$$ with $y(a)= c_6(T-a)^{(2\alpha-1)}$ we have
\begin{eqnarray*}
y(t)&=&\bigg\{\frac{(1-\beta)\lambda^2L_\sigma^2(T-a)^{(2\alpha-1)}}{(2\alpha-1)(1-\beta)}(t-a)^{(2\alpha-1)(1-\beta)}+y(a)^{1-\beta}\bigg\}^{\frac{1}{1-\beta}}
\end{eqnarray*}

and the solution fails to exist for all $(2\alpha-1)(1-\beta)>0$, that is, for all $\beta>1$ and $\alpha<\frac{1}{2}$.
\end{proof}

\begin{theorem}
Suppose that Conditions \ref{cond:E-U1} and \ref{cond:E-U3} are in force. Then there does not exist a solution to Equation \eqref{eqn1} for all $\alpha\geq \frac{1}{2}.$
\end{theorem}
\begin{proof}Assume that the initial condition $u_0(x)>c_5$, then by Condition \ref{cond:E-U3}, and applying Jensen's inequality we have
\begin{eqnarray*}
\mathbb{E}|u(x,t)|^2&\geq& |u_0(x)|^2+\lambda^2\int_a^t (s-a)^{2(\alpha-1)}\mathbb{E}|\sigma(u(x,s))|^2d s\\
&\geq& c_5^2+\lambda^2L_\sigma^2\int_a^t (s-a)^{2(\alpha-1)}\mathbb{E}|u(x,s)|^{2(\beta)}d s\\
&\geq& c_6+\lambda^2L_\sigma^2\int_a^t (s-a)^{2(\alpha-1)}\big(\inf_{x\in\mathbb{R}}\mathbb{E}|u(x,s)|^2\big)^{\beta}ds.
\end{eqnarray*}
Let $f(t):=\displaystyle\inf_{x\in\mathbb{R}}\mathbb{E}|u(x,t)|^2$. Then it follows that
\begin{eqnarray}
f(t)&\geq& c_6+\lambda^2L_\sigma^2\int_a^t (s-a)^{2(\alpha-1)}f^{\beta}(s)d s.\label{non-linear-ineq}
\end{eqnarray}
Now consider the equation
$$ \frac{f'(t)}{f^\beta(t)}=\lambda^2 L_\sigma^2 (t-a)^{2(\alpha-1)}   $$
with $f(a)=c_6$. This equation has a solution
$$
f^{1-\beta}(t)=c_6^{1-\beta} +(1-\beta )(\lambda^2L_\sigma^2/(2\alpha-1))(t-a)^{2\alpha-1}.
$$

This solution  blows up in finite time. By comparison principle $f$ blows up in finite time too.

When $\alpha=1/2$ from \eqref{non-linear-ineq} we get  for  some $b>a$
\begin{eqnarray}
f(t)&\geq& c_6+\lambda^2L_\sigma^2\int_b^t (s-a)^{-1}f^{\beta}(s)d s.\label{non-linear-ineq-2}
\end{eqnarray}
Considering  the equation
$$ \frac{f'(t)}{f^\beta(t)}=\lambda^2 L_\sigma^2 (t-a)^{-1}   $$
with $f(b)=c_6$. This equation has a solution
$$
f^{1-\beta}(t)=c_6^{1-\beta} +(1-\beta )(\lambda^2L_\sigma^2)\ln \bigg[(t-a)/(b-a)\bigg].
$$

This solution  blows up in finite time. By comparison principle $f$ blows up in finite time too.
\end{proof}

\section{Conclusion}\label{400} We studied a class of conformable time-fractional stochastic equation, proved the existence and uniqueness result  under some suitable conditions  on the initial function and also studied the asymptotic behaviour (long-term properties) of the solution with respect to the time $t$ and noise parameter $\lambda$. The obtained result also shows that the global non-existence of the solution depends on the parameter $\alpha$, that is, the second moment of the solution fails to exist for all $t$ when the non-linearity term $\sigma$ grows faster than linear  for  $\alpha \in (0,1)$.
\begin{small}
\end{small}

\begin{thebibliography}{99}
\bibitem{Abdeljawad}T. Abdeljawad.\,\,\emph{On conformable fractional calculus.} \,J. Comput. Appl. Math., \,{\bf 279} 57--66 (2015).
%
\bibitem{Al-Refai}M. Al-Refa, T. Abdeljawad.\,\,\emph{Fundamental Results of conformable Strum--Liouville Eigenvalue Problems.}\, Complexity, {\bf 2017} Art. ID 3720471,\,1--7 (2017).
%
\bibitem{Asawasamrit}S. Asawasamrit, S. K. Ntouyas,  P. Thiramanus, J. Tariboon.\,\,\emph{Periodic boundary value problems for impulsive conformable fractional integro-differential equations.}\, Bound. Value Probl., \,{\bf 2016:122}\,  (2016).

%
\bibitem{Birgani} O. T. Birgani, S. Chandok, N. Dedovi\'c, S. Radenovi\'c.\,\,\emph{A note on some recent results of the conformable derivative.}\,  Adv. Theory Nonlinear Anal. Its Appl., \, {\bf 3}(1)\, 11--17 (2019).
%
\bibitem{Cenesiz}Y. \c{C}enesiz, A. Kurt, E. Nane.\,\,\emph{Stochastic solutions of Conformable fractional Cauchy problems.}\, Stat. Probabil. Lett., {\bf 124}\, 126--131 (2017).
%
\bibitem{Foondun}M. Foondun, W. Liu, K. Tian.\,\,\emph{Moment bounds for  a class of Fractional Stochastic Heat Equations.}\, Ann. Probab.,\, {\bf 45}(4) \,2131--2153 (2017).
%
\bibitem{Foondun1}M. Foondun, W. Liu, M. E. Omaba.\,\,\emph{On Some Properties of  a class of Fractional Stochastic Heat Equations.}\, J. Theoret. Probab., \,{\bf 30}(4) \,1310--1333 (2017).
%
\bibitem{Foondun2}M. Foondun, E. Nane.\,\,\emph{Asymptotic properties of some space-time fractional stochastic equations.}\, Math. Z.,\, {\bf 287}(1--2) \,493--519 (2017).
%
\bibitem{Foondun3}M. Foondun, W. Liu, E. Nane.\,\,\emph{Some non-existence results for a class of stochastic partial differential equations.}\, J. Differential Equations, \,{\bf 266}(5) \,2575--2596 (2019).
%
\bibitem{Gholami}Y. Gholami, K. Ghanbari.\,\,\emph{New class of conformable derivatives and applications to differential impulsive systems.}\, S\'{e}MA Journal, \,{\bf 75}(2) \,305--333 (2018).
%
\bibitem{Iyiola} O. S. Iyiola, E. R. Nwaeze.\,\,\emph{Some new results on the new conformable fractional calculus with application using D'Alambert approach.} \,Progr. Fract. Differ. Appl.,\, {\bf 2}(2) 115--121  (2016).
%
\bibitem{Kaplan}M. Kaplan, A. Akbulut.\,\,\emph{Application of two different algorithms to the approximate long water wave equation with conformable fractional derivative.}\, Arab J. Basic Appl. Sci., \,{\bf 25}(2) \,77--84 (2018).
%
\bibitem{Khalil}R. Khalil, M. A. Horani, A. Yousef, M. Sababheh.\,\,\emph{A new definition of fractional derivative.}\, J. Comput. Appl. Math.,\, {\bf 264}  \,65--70 (2014).

%
\bibitem{Lakshmikantham}V. Lakshmikantham, A. S. Vatsala.\,\,\emph{Basic theory of fractional differential equations.} \,Nonlinear Anal., {\bf 69}(8) \,2677--2682 (2008).
%
\bibitem{Meng}S. Meng, Y. Cui.\,\,\emph{The Extremal Solution To Conformable Fractional Differential Equations Involving Integral Boundary Condition.}\, Mathematics, \,{\bf 7}(2) 186 (2019).
%
\bibitem{Mijena}J. Mijena, E. Nane.\,\,\emph{Space-time fractional stochastic partial differential equations.}\, Stoch. Process. Appl., {\bf 159}(9) 3301--3326 (2015).
%
\bibitem{Omaba} M. E. Omaba.\,\, \emph{Some properties of a class of stochastic heat equations,} Ph.D Thesis, Loughborough University, UK, (2014).
%
\bibitem{Omaba1}M. E. Omaba.\,\,\emph{On Space-Time Fractional Heat Type Non-Homogeneous Time-Fractional Poisson Equation.}\, J. Adv. Math. Comput. Sci., \,{\bf 28}(4) \,1--18 (2018).
%
\bibitem{Omaba2}M. E. Omaba.\,\,\emph{On Space-Fractional Heat Equation with Non-homogeneous Fractional Time Poisson Process.}\, Progr. Fract. Differ. Appl., \,{\bf 6}(2) \,1--13 (2019).
%
\bibitem{Omaba3}M. E. Omaba, E. R. Nwaeze.\,\,\emph{Moment Bound of Solution to a Class of Conformable Time-Fractional Stochastic Equation.}\, Fractal  Fract., \,{\bf 3}(18) \,1--13 (2019).
%
\bibitem{Omaba4}M. E. Omaba, E. Nwaeze, L. O. Omenyi.\,\,\emph{On non-existence of global weak-predictable random field solutions to a class of SHEs.}\, Asian Res. J. Math.,\, {\bf 4}(2) 1--14 (2017).
%
\bibitem{Sakthivel}R. Sakthivel, P. Revathi, Y. Ren.\,\,\emph{Existence of solutions for nonlinear fractional stochastic differential equations.}\, Nonlinear Anal., \,{\bf 81} \,70--86 (2013).
%
\bibitem{Salahshour}S. Salahshour, A. Ahmadian, F. Ismail, D. Baleanu, N. Senu.  \,\,\emph{A New fractional derivative for differential equation of fractional order under internal uncertainty.}\,  Adv. Mech. Eng.,\,{\bf 7}(12) \,1--11 (2015).

%
\bibitem{Souahi}A. Souahi, A. B.  Makhlouf, M. A. Hammami.\,\,\emph{Stability analysis of conformable fractional-order nonlinear systems.}\, Indag. Math., \,{\bf 28}\, 1265--1274 (2017).

%
\bibitem{Usta}F. Usta.\,\,\emph{A conformable calculus of radial basis functions and its applications.}\, Int. J. Optim. Control Theor. Appl., \,{\bf 8}(2) \,178--182 (2018).
%
\bibitem{Walsh} J. B. Walsh.\,\, \emph{An introduction to stochastic partial differential equations, In Lecture Notes in Maths 1180}, Springer, Berlin,\,265--439 (1986).
%
\bibitem{Yang}S. Yang, L. Wang, S. Zhang.\,\,\emph{Conformable derivative:\,\,Application to non-Darcian flow in low-permeability porous media.}\,  Appl. Math. Lett., \,{\bf 79} \,105--110 (2018).
%
\bibitem{Zhang} X. Zhang, P. Agarwal, Z. Liu, H. Peng, F. You, Y. Zhu. \,\,\emph{Existence and uniqueness of solutions for stochastic differential equations of fractional-order $q>1$ with finite delays.}\, Adv. Differ. Equ., \,{\bf 2017:123},\,1--18 (2017).
%
\bibitem{Zhao}D. Zao, M. Luo.\,\,\emph{General conformable fractional derivative and its physical interpretation.}\, Calcolo, {\bf 54}(3) 903--917 (2017).
%
\bibitem{Zhong}W. Zhong, L. Wang.\,\,\emph{Basic theory of initial value problems of conformable fractional differential equations.}\, Adv. Differ. Equ., {\bf 2018:321} 1--14 (2018).
\end{thebibliography}
 \end{document}